\documentclass{amsart}

\usepackage{amssymb}
\usepackage{color}
\usepackage{tikz-cd}
\usepackage{stmaryrd}
\usepackage{hyperref}
\usepackage{ulem}
\hypersetup{
    colorlinks=true,
    linkcolor=black,
    filecolor=blue,      
    urlcolor=blue,
    citecolor=blue,
    pdfpagemode=FullScreen,
    }
\usepackage[margin=1.5 in, top=1.3in, bottom=1.3 in]{geometry}
\usepackage{enumitem}

\newtheorem{thm}{Theorem}[section]
\newtheorem{lem}[thm]{Lemma}

\newtheorem{prop}[thm]{Proposition}

\newtheorem{thm*}{Theorem}

\theoremstyle{definition}

\theoremstyle{remark}
\newtheorem{remark}[thm]{Remark}

\numberwithin{equation}{section}

\newcommand{\R}{{\mathbb{R}}}
\newcommand{\Z}{{\mathbb{Z}}}
\newcommand{\C}{{\mathbb{C}}}
\newcommand{\Q}{{\mathbb{Q}}}

\newcommand{\overbar}{\overline}
\newcommand{\om}{\omega}

\newcommand{\fix}{\mathrm{Fix}}

\DeclareMathOperator{\supp}{\mathrm{supp}}

\DeclareMathOperator{\Ham}{\mathrm{Ham}}

\DeclareMathOperator{\Symp}{\mathrm{Symp}}

\DeclareMathOperator{\hn}{\mathrm{HN}}

\DeclareMathOperator{\hfn}{\mathrm{HFN}}

\DeclareMathOperator{\h}{\mathrm{H}}

\DeclareMathOperator{\flux}{\mathrm{Flux}}

\DeclareMathOperator{\cz}{\mathrm{CZ}}

\newcommand\blfootnote[1]{
	\begingroup
	\renewcommand\thefootnote{}\footnote{#1}
	\addtocounter{footnote}{-1}
	\endgroup
}

\newcommand{\mbmail}{marta.batoreo@ufes.br}
\newcommand{\bfmail}{brayan.ferreira@ufes.br}
\newcommand{\msaemail}{msatallah@ime.usp.br}

\begin{document}

\title{Remarks on Periodic Points, Flux, and the Mean Index in Symplectic Dynamics}

\author{Marcelo S. Atallah}
\address{Marcelo S. Atallah, Universidade de São Paulo, Departamento de Matemática/IME , Rua do Matão, 1010 – São Paulo – SP – Brasil, 05508-090}
\email{\msaemail}

\author{Marta Batoréo}
\address{Marta Batoréo, Universidade Federal do Espírito Santo, Departamento de Matemática/CCE, 
Av. Fernando Ferrari, 514, Goiabeiras – Vitória – ES, Brasil, 29075-910}
\email{\mbmail}

\author{Brayan Ferreira}
\address{Brayan Ferreira, Universidade Federal do Espírito Santo, Departamento de Matemática/CCE, 
Av. Fernando Ferrari, 514, Goiabeiras – Vitória – ES, Brasil, 29075-910}
\email{\bfmail}


\begin{abstract}
It is shown that a symplectic isotopy with non-degenerate end-point on a closed symplectically Calabi--Yau manifold has simple periodic points of arbitrarily large period under a simple non-vanishing condition on the Morse--Novikov homology associated with its flux. The proof combines a sharp estimate for the Conley--Zehnder index under the assumption that the mean index is an even integer with a classical asymptotic iteration argument in the non-Hamiltonian Floer-theoretic setting. We also apply this estimate to the symplectic dynamics of the standard symplectic $4$-torus.

\end{abstract}

\maketitle

\section{Introduction and main results}
Two well-known conjectures in Hamiltonian dynamics concern the existence of infinitely many periodic points. The Conley conjecture asserts that, for a large class of symplectic manifolds, every Hamiltonian diffeomorphism has infinitely many periodic points. For a complementary class of symplectic manifolds, the Hofer--Zehnder conjecture predicts the same conclusion provided that the number of fixed points exceeds that required by the Arnol'd conjecture. See \cite{ginzburg2010conley, ginzburg2009action, ginzburg2012conley, ginzburg2019conley} and \cite{shelukhin2022hofer, atallah2025hz, bai2024hz} for recent result


By contrast, in symplectic (non-Hamiltonian) dynamics, there are many examples of symplectic diffeomorphisms, isotopic to the identity, with any prescribed number of periodic points \cite{LeCalvez_conservative2022, atallah2024number}. On the other hand, for a closed surface of positive genus $g$, a symplectic diffeomorphism isotopic to the identity and having more than $2g-2$ (Floer-theoretically visible) fixed points must possess simple periodic points of arbitrarily large period \cite{le2006periodic, atallah2024number}. 


The objective of this work is to present a class of symplectic manifolds for which a non-Hamiltonian variant of the Conley conjecture holds. Our main results establish, in particular, that any non-degenerate symplectic diffeomorphism, which is isotopic to the identity, on the product of a closed surface of genus $g\geq2$ and a closed symplectically Calabi--Yau manifold with trivial first Betti number, has simple periodic points of arbitrarily large period. This follows from a more general statement proving the same conclusion for any non-degenerate symplectic diffeomorphism $\psi$ on a closed symplectically Calabi--Yau manifold such that the Morse--Novikov homology associated with the flux of an isotopy $\psi_t$ with endpoint $\psi$ is non-trivial in degree $2n-1$. 

Our approach also yields a result on the standard symplectic 4-torus. Namely, it is shown that any non-degenerate symplectic diffeomorphism admitting a fixed point with even mean index possesses simple periodic points of arbitrarily large period. This result is closer in spirit to the Hofer--Zehnder conjecture.

\subsection{Morse--Novikov homology of {the} flux and infinitely many periodic points} 

Let $(M,\omega)$ be a closed symplectic manifold and $\psi$ a symplectic diffeomorphism isotopic to the identity. Given a symplectic isotopy $\psi_t \in \Symp(M,\omega)$, $t \in [0,1]$, based at the identity and with endpoint $\psi_1=\psi$, let $X_t$ be the time-dependent vector field defined by $\frac{d}{dt}\psi_t = X_t \circ \psi_t.$ The $1$-form $\imath_{X_t}\omega$ is closed, and the flux homomorphism is defined by

\begin{align*}
\flux \colon \widetilde{\Symp_0}(M,\om) \to \h^1(M;\R),\quad[\psi_t] \mapsto \left[\int_0^1 \omega(\cdot, X_t)dt\right]
\end{align*}

on the universal cover of the identity component $\Symp_0(M,\om)$ of the group of symplectomorphisms of $M$. The kernel of this homomorphism is the universal cover of the group of Hamiltonian diffeomorphisms $\Ham(M,\omega)$. Recall that $\psi$ is a Hamiltonian diffeomorphism precisely when the isotopy $\psi_t$ can be chosen so that $\omega(\cdot, X_t) = dH_t$ for a Hamiltonian $H \colon \R \times M \to \R$.

Given a fixed point $x$ of $\psi$, we say that $x$ is contractible with respect to the isotopy $\psi_t$ if the loop $\{\psi_t(x)\}$ is contractible in $M$. We denote the set consisting of these points by $\fix_0(\psi)$. Our first result guarantees the existence of simple periodic points of $\psi$ of arbitrarily large prime period $p$ that are contractible with respect to $\psi_t^p$, under a non-vanishing condition on the Morse--Novikov homology associated with $\flux(\psi_t) \in \h^1(M;\R)$.

\begin{thm}\label{thm:main}
Let $(M,\omega)$ be a closed symplectically Calabi--Yau manifold of dimension $2n$, where $n\geq2$. Suppose that $\psi_t$ is a symplectic isotopy of flux $\theta \in \h^1(M;\R)$, based at the identity and with non-degenerate endpoint $\psi$. If
\begin{equation}\label{eq:main}
\supp \hn_*(\theta) \cap \{0,1,2n-1,2n\} \neq \emptyset,
\end{equation}
 then $\psi$ has a simple $p$-periodic point which is contractible with respect to $\psi_t^p$, for each sufficiently large prime $p$.
\end{thm}

The proof of Theorem \ref{thm:main} relies on a refinement of the standard comparison between the mean index and the Conley--Zehnder index in the special case of vanishing mean index; see Theorem \ref{thm:zeromean} below. This allows us to apply a classical asymptotic iteration argument without relying on the presence of the fundamental class, which never exists in the non-Hamiltonian case. In fact, the argument only requires that the support of the Floer--Novikov homology of $\psi_t$ is not contained in the interval $[-(n-2),n-2]$. This follows from \eqref{eq:main}; see Section \ref{sec:proof} for details.

\begin{remark}\label{rmk:Hamiltoniancase}
In the case where the isotopy $\varphi_t^H$ is Hamiltonian, the hypothesis \eqref{eq:main} is automatically satisfied. Indeed, $$\hn_{2n}(\flux(\varphi_t^H)) = \hn_{2n}([dH_t]) = \h_{2n}(M) \neq 0.$$  In this setting, the weakly-non-degenerate Hamiltonian case was shown in \cite{salamon1992morse}.
\end{remark}

As {{a}} consequence of Theorem \ref{thm:main}, we find a large class of examples of manifolds for which a non-Hamiltonian version of the Conley conjecture holds.

\begin{thm}\label{thm:examples}
Let $(\Sigma,\omega_{\Sigma})$ be a closed surface of genus $g\geq 2$, and let $(Q^{2n-2},\omega_Q)$ be a closed symplectically Calabi--Yau manifold such that {the first Betti number of $Q$, $b_1(Q),$ is zero}. Then, any non-degenerate symplectomorphism $\psi \in \Symp_0(\Sigma \times Q, \omega_{\Sigma}\oplus\omega_Q)$ has a simple $p$-periodic point for each sufficiently large prime $p$.
\end{thm}

Examples of such manifolds $Q$ are given by $K3$ surfaces and by complex degree $d=n+1$ smooth hypersurfaces in $\mathbb{C}P^n$, with $n\geq 3$. The latter are simply connected, since $n\geq 3$, and symplectically Calabi--Yau, since the first Chern class is $(n+1-d)h$, where $h$ is the hyperplane class in $H^2(\mathbb{C}P^n,\mathbb{Z})$.

\begin{proof}[Proof of Theorem \ref{thm:examples}]
Let $\psi_t$ be any symplectic isotopy based at the identity with endpoint $\psi$. Suppose $\psi_t$ has flux $$\theta=(\theta_{\Sigma}, 0_Q)\in \h^1(\Sigma\times Q;\R) = \h^1(\Sigma;\R).$$

By the K\"{u}nneth formula, we have
\begin{align*} 
\hn_{2n-1}(\theta)=\hn_1(\theta_\Sigma) \otimes \h_{2n-2}(Q) \oplus \hn_2(\theta_\Sigma) \otimes \h_{2n-3}(Q) \neq 0.
\end{align*}
Here we used the facts that $\h^1(Q;\R) = 0$, $\hn_*(0_Q) = \h_*(Q)$ and $\hn_1(\theta_\Sigma) \neq 0$; see Proposition \ref{prop:HNranksurfaces} below. Moreover, we have $\h_{2n-2}(Q) \neq 0$ because of the existence of a fundamental class. Therefore, by Theorem \ref{thm:main}, $\psi$ has a simple $p$-periodic point which is contractible with respect to $\psi_t^p$, for each sufficiently large prime $p$.
\end{proof}

\begin{remark}
The same argument proves the stronger statement that the conclusion continues to hold when $b_1(Q)\neq0$ provided that the $Q$-projection of the flux $\theta_Q$ is trivial.
\end{remark}

\begin{remark}
In the setting of the proof, the flux group is always trivial. In particular, the homotopy class of the orbits, i.e. $\{\psi_t(x)\}, x\in\fix(\psi)$ is independent of the isotopy. 
\end{remark}

On the other hand, in $\Sigma\times S^2$, there are many examples of symplectic diffeomorphisms isotopic to the identity with only finitely many periodic points. For instance, one may take the product of a non-degenerate symplectic diffeomorphism $\psi_\Sigma$ on $\Sigma$ having finitely many periodic points, as in \cite{atallah2024number}, with an irrational rotation. Thus, the Calabi--Yau hypothesis cannot be removed. Similarly, on $\Sigma\times\mathbb{T}^2$, one may take the product of $\psi_\Sigma$ with an irrational translation of $\mathbb{T}^2$, showing that hypothesis \eqref{eq:main} is also necessary. 

\subsection{An application in dimension $4$}

An interesting consequence of our estimate in dimension $4$ is that when the mean index is an even integer, it equals the Conley–Zehnder index; see Theorem \ref{thm:zeromean} below. This observation allows us to prove that the existence of a fixed point with even mean index forces the existence of simple periodic points of arbitrarily large period.

\begin{thm}\label{thm:t4}
Let $\psi \in \Symp_0(\mathbb{T}^4,\omega)$ be non-degenerate. Suppose that $\psi$ has a fixed point $x_0$ such that its mean index is an even integer. Then $\psi$ has a simple $p$-periodic point for every sufficiently large prime $p$.
\end{thm}

\begin{remark}
The same proof applies to any closed $4$-dimensional symplectically Calabi--Yau manifold, provided that $\psi$ is the endpoint of a symplectic isotopy $\psi_t$ such that $\Delta(\psi_t,x_0)$ is even and the Morse--Novikov homology associated with its flux vanishes in degree $2$. By the K\"{u}nneth formula, examples include products $\Sigma\times\mathbb{T}^2$, where $\Sigma$ is a closed surface of positive genus.
\end{remark}

Finally, we note that, on $\mathbb{T}^{2n}$, the Morse--Novikov homology associated with any non-exact closed $1$-form vanishes. It is therefore natural to expect that the existence of a non-degenerate fixed point is already sufficient to obtain the same conclusion.

\subsection{Acknowledgments} We thank Alberto Abbondandolo and Pietro Majer for sharing unpublished notes containing the description and properties of the mean index and the Conley–Zehnder index we use in this work. We also thank Viktor Ginzburg, Başak G\"urel, and Egor Shelukhin for their interest, suggestions and encouragement in this project. This work began at the \textit{International Conference on Symplectic Dynamics} at SUSTech, Shenzhen, and was partially developed during a visit of M.S.A. to M.B. and B.F. at UFES. We are grateful to both organizations for their support. \\ B.F. is grateful for the hospitality of the IME-USP, where part of this work was completed.

\section{Morse--Novikov homology and Floer--Novikov homology}
In this section, we briefly recall the properties of Morse--Novikov homology and Floer--Novikov homology that will be used in this work. We closely follow \cite{atallah2024number}, to which we refer the reader for further details.

\subsection{Morse--Novikov Homology}
Let $\theta$ be a closed $1$-form on a smooth manifold $M$, and let $\pi \colon \overline{M} \to M$ be the covering associated with the kernel of the period homomorphism
$$
I_\theta\colon \pi_1(M) \to \R,\qquad [c] \mapsto \int_c \theta.
$$
Then $\pi^*\theta$ is exact on $\overline{M}$, and the group of deck transformations is $\Gamma_\theta=\pi_1(M)/\ker I_\theta.$ Assume that $\pi^*\theta=df_\theta$ for a Morse function $f_\theta\colon\overline{M}\to\R$. The Morse--Novikov homology $\hn(\theta)$ is a $\Q$-vector space and a finitely generated free module over the Novikov ring
\begin{align*}
	\Lambda_{\theta}
	=
	\Big\{
	\sum_i a_i g_i
	\ \vert\
	a_i\in\Q,\ g_i\in\Gamma_\theta,\ \text{satisfying (*)}
	\Big\},
\end{align*}
where (*) requires that, for every $c\in\R$, the set $\{i\ \vert\ a_i\neq0,\ I_\theta(g_i)>c\}$ be finite.

The homology $\hn(\theta)$ depends only on the cohomology class $[\theta]\in\h^1(M;\R)$. More precisely, if $[\theta_1]=[\theta_2]$, then $\hn_k(\theta_1)=\hn_k(\theta_2)$ as modules over $\Lambda_{\theta_1}=\Lambda_{\theta_2}$ for every $k$. If $\theta$ is exact, then $\hn_*(\theta)$ recovers the singular homology $\h_*(M)$ with rational coefficients.

The Euler characteristic $\chi(\theta):=\sum_i(-1)^i\dim_{\Lambda_\theta}\hn_i(\theta)$ coincides with $\chi(M)$; see, for example, \cite[Appendix C]{van1995symplectic}. For non-exact forms, we have the following:

\begin{prop}[{\cite[Proposition 4.12]{ono2006floer}}]\label{prop:vanishHN}
	Let $\theta$ be a closed non-exact $1$-form on a closed smooth manifold $N$. Then
	$
	\dim_{\Lambda_\theta}\hn_k(\theta)=0
	$
	for $k=0$ and $k=\dim N$.
\end{prop}

Combining Proposition \ref{prop:vanishHN} with the identity $\chi(\theta)=\chi(\Sigma)$ gives the following computation for closed surfaces.

\begin{prop}\label{prop:HNranksurfaces}
	Let $\Sigma$ be a closed orientable surface of genus $g\geq1$, and let $\theta$ be a closed non-exact $1$-form on $\Sigma$. Then $$\dim_{\Lambda_\theta}\hn_*(\theta)=\dim_{\Lambda_\theta}\hn_1(\theta)=-\chi(\Sigma)=2g-2.$$
\end{prop}

\subsection{Floer--Novikov Homology}
Let $(M,\omega)$ be a closed symplectic manifold and $\psi\in\Symp_0(M,\om)$. Given a symplectic isotopy $\psi_t$, $t\in[0,1]$, based at the identity and with endpoint $\psi_1=\psi$, let $\flux(\psi_t)\in\h^1(M;\R)$ denote its flux. The Floer--Novikov homology $\hfn(\psi_t)$ is a semi-infinite-dimensional analogue of Morse homology associated with a suitable action functional.

 Let $\theta$ be a closed $1$-form representing $\flux(\psi_t)$. After composing with a loop of Hamiltonian diffeomorphisms if necessary, we may assume that
$
\omega(\cdot,X_t)=\theta+dH_t=:\theta_t,
$
where $H_t\colon M\to\R$, $t\in S^1=\R/\Z$, is a periodic Hamiltonian; see \cite[Lemma 2.1]{van1995symplectic}. The set $\mathcal{P}(\theta_t)$ of $1$-periodic solutions of $\dot{x}=X_t(x)$ is in bijection with the fixed points of $\psi=\psi_1$. An orbit $x\in\mathcal{P}(\theta_t)$ is \textit{non-degenerate} if the linear symplectic map
$
d\psi_{x(0)}\colon T_{x(0)}M\to T_{x(0)}M
$
does not have $1$ as an eigenvalue. The diffeomorphism $\psi$ is \textit{non-degenerate} if all such orbits are non-degenerate. In this case, compactness of $M$ implies that $\mathcal{P}(\theta_t)$ is finite.


Let $\mathcal{L}_0M$ denote the component of contractible loops. Define
$
\mathcal{I}_\theta=I_\theta\circ e_*\colon\pi_1(\mathcal{L}_0M)\to\R,
$
where $e_*\colon\pi_1(\mathcal{L}_0M)\to\pi_1(M)$ is induced by the evaluation map $e(x)=x(0)$. We also consider
\begin{align*}
	\mathcal{I}_{\omega}\colon\pi_1(\mathcal{L}_0M)\to\R,\qquad
	\{x_s\}\mapsto\int_{C(\{x_s\})}\omega,
\end{align*}
where $C(\{x_s\})$ is the ``torus'' swept out by $\{x_s\}$:
$
(s,t)\in S^1\times S^1\mapsto x_s(t)\in M.
$
Let
$
\Pi\colon\tilde{\mathcal{L}}\overbar{M}\to\mathcal{L}_0M
$
be the covering associated with
$
\ker\mathcal{I}_\om\cap\ker\mathcal{I}_\theta
\subset\pi_1(\mathcal{L}_0M).
$
Its points can be identified with equivalence classes $[x,u]$, where $x\in\mathcal{L}_0M$ and $u$ is a capping of $x$. Two pairs $(x,u)$ and $(y,v)$ are equivalent if and only if $x=y$ and
\begin{align*}
	\int_u\om=\int_v\om
	\quad\text{and}\quad
	\int_{l_{u,0}}\theta=\int_{l_{v,0}}\theta.
\end{align*}
Here, $l_{u,t}(s):=u(se^{2\pi it})$ for $s\in[0,1]$, and similarly for $l_{v,t}$. The deck transformation group is
$
\pi_1(\mathcal{L}_0M)/
\bigl(\ker\mathcal{I}_\om\cap\ker\mathcal{I}_\theta\bigr).
$
The action functional
$
\mathcal{A}_{\psi_t}\colon
\tilde{\mathcal{L}}\overbar{M}\to\R
$
is a distinguished primitive of the $1$-form $\alpha_{\psi_t}$.
Its critical points are the lifts to $\tilde{\mathcal{L}}\overbar{M}$ of the orbits in $\mathcal{P}(\theta_t)$. Roughly speaking, $\hfn(\psi_t)$ is the Floer homology of $\mathcal{A}_{\psi_t}$. The associated complex is generated by the classes $[x,u]$ and is a $\Q$-vector space and a finitely generated free module over
\begin{align*}
	\Lambda_{\om,\theta}
	=
	\Big\{
	\sum_i a_i g_i
	\ \vert\
	a_i\in\Q,\ g_i\in\Gamma_{\theta,\om},\
	\text{satisfying (**)}
	\Big\}.
\end{align*}
Condition (**) requires that the set
$
\bigl\{
i\ \vert\ a_i\neq0,\
(\mathcal{I}_\om+\mathcal{I}_\theta)(g_i)>c
\bigr\}
$
be finite for all $ c\in\R$.

The Floer--Novikov homology depends only on the flux among isotopies with fixed endpoints. Namely, if $\psi_t$ and $\psi_t^\prime$ are based at the identity and have the same endpoint and flux, then
$
\hfn(\psi_t)=\hfn(\psi_t^\prime).
$
The following theorem relates the rank of this homology with that of the Morse--Novikov homology of the flux form.

\begin{thm}[{\cite[Theorem 3.1]{van2020floer}}]\label{thm:hfnandhn}
	Let $\psi_t\in\Symp_0(M,\om)$ be a symplectic isotopy such that $\psi_0=id$ and $\psi_1$ is non-degenerate. Suppose $\mathrm{Flux}(\psi_t)=[\theta]$. Then
	$$
	\mathrm{rank}_{\Lambda_{\om,\theta}}\hfn(\psi_t)
	=
	\dim_{\Lambda_\theta}\hn(\theta).
	$$
\end{thm}

\section{Mean index and Conley--Zehnder index}
We briefly recall the definitions and relevant properties of the mean index and the Conley--Zehnder index, following \cite{salamon1992morse}.

Let $\mathrm{Sp}(2n)$ denote the group of $2n\times2n$ symplectic matrices. For $A\in\mathrm{Sp}(2n)$, denote its eigenvalues of the first kind\footnote{That is, $\vert\lambda_j\vert<1$, or $\vert\lambda_j\vert=1$ and $\lambda_j$ has positive Krein signature.} by
$
\lambda_1,\ldots,\lambda_n.
$
The characteristic exponents of $A$ are real numbers $\theta_j$ satisfying
$
e^{i\pi\theta_j}
=
{\lambda_j}/{\vert\lambda_j\vert}.
$
It is possible to choose their determinations in $[0,2]$ and order them increasingly. If $\det(I-A)>0$, an even number $2k$ of the characteristic exponents are even integers, and we set
\[
\theta_j=0,\quad j=1,\ldots,k,
\qquad\text{and}\qquad
\theta_j=2,\quad j=n-k+1,\ldots,n.
\]
If $\det(I-A)<0$, an odd number $2k+1$ of them are even integers, and we set
\[
\theta_j=0,\quad j=1,\ldots,k+1,
\qquad\text{and}\qquad
\theta_j=2,\quad j=n-k+1,\ldots,n.
\]
These conventions uniquely determine the characteristic exponents and yield continuous functions $\theta_j:\mathrm{Sp}(2n)^*=\{ A \in \mathrm{Sp}(2n);\det(A-I) \neq 0\} \rightarrow[0,2]$, $A\mapsto \theta_j(A)$.

Let $\Psi\colon[0,1]\to\mathrm{Sp}(2n)$ be a smooth path. Recall that $\pi_1(\mathrm{Sp}(2n))=\Z$ and that there is a continuous map
$
\rho\colon\mathrm{Sp}(2n)\to S^1
$
inducing this isomorphism and restricting to the complex determinant on
$
\mathrm{U}(n)\cong\mathrm{Sp}(2n)\cap\mathrm{O}(2n).
$
Choose a lift $\alpha\colon[0,1]\to\R$ such that
$
\rho(\Psi(t))=e^{i\alpha(t)}.
$
The mean index of $\Psi$ is
\[
\Delta(\Psi)
=
\frac{\alpha(1)-\alpha(0)}{\pi}.
\]

We will also use the following description. The eigenvalues of the first kind of $\Psi(t)$ can be represented by continuous functions
$
\lambda_j\colon[0,1]\to\C,\, j=1,\ldots,n.
$
Let $\theta_j\colon[0,1]\to\R$ be corresponding characteristic exponents satisfying
$
e^{i\pi\theta_j(t)}
=
{\lambda_j(t)}/{\vert\lambda_j(t)\vert}.
$ Then
\begin{equation}\label{eq:mean index}
\Delta(\Psi)=\sum_{j=1}^n(\theta_j(1)-\theta_j(0)).
\end{equation}
Indeed,
\[
e^{i\alpha(t)}
=
\rho(\Psi(t))
=
\prod_{j=1}^n
\frac{\lambda_j(t)}{\vert\lambda_j(t)\vert}
=
e^{i\pi\sum_{j=1}^n\theta_j(t)},
\]
which gives the claimed formula. In this case, we denote $\theta_j(\Psi):=\theta_j(1)$.

Furthermore, if  $\Psi\colon[0,1]\to\mathrm{Sp}(2n)^*$, then
\begin{equation}\label{eq:mean index2}
\Delta(\Psi)=\sum_{j=1}^n(\theta_j(\Psi(1))-\theta_j(\Psi(0))).
\end{equation}
Indeed, as a continuous function on $\mathrm{Sp}(2n)^*$, $\sum\theta_j$ is a lift of $\rho$. 

Now let $x\in\mathcal{P}(\theta_t)$ be a contractible periodic orbit and let $u\colon D^2\to M$ be a capping of $x$, so that $u|_{\partial D^2}=x$. The capping induces a symplectic trivialization of $x^*TM$, under which the linearized isotopy
$
(d\psi_t)_{x(0)}
\colon
T_{x(0)}M\to T_{\psi_t(x(0))}M
$
becomes a smooth path
$
\Psi_x\colon[0,1]\to\mathrm{Sp}(2n)
$
starting at the identity. We define
$
\Delta(\psi_t,x,u):=\Delta(\Psi_x).
$ When $x$ is non-contractible, we fix instead a symplectic trivialization $\tau$ of $x^*TM$ and define in the same way $\Delta(\psi_t,x,\tau) = \Delta(\Psi_x)$ .

If $x$ is non-degenerate, then $\Psi_x(1) \in \mathrm{Sp}(2n)^*$. It can therefore be connected through non-degenerate symplectic matrices to $W_+=-I$ or $ W_-=\mathrm{diag}(2,-1,\ldots,-1,1/2,-1,\ldots,-1),$
depending on whether $\det(I-\Psi_x(1))$ is positive or negative. Concatenating $\Psi_x$ with such a path gives a path $\Psi_x^\prime$, and we define $\cz(\psi_t,x,u):=\Delta(\Psi_x^\prime).$
This definition is independent of the choice of connecting path.

We next express the Conley--Zehnder index in terms of the characteristic exponents. Define $[[\cdot]]\colon\R\to\Z$ by setting $[[r]]=r$ when $r$ is an integer and taking $[[r]]$ to be the nearest odd integer to $r$ otherwise. Moreover, we have
\begin{equation}\label{eq:theta}
\theta_1(W_+)=\ldots=\theta_n(W_+)=1,
\qquad
\theta_1(W_-)=0,\quad
\theta_2(W_-)=\ldots=\theta_n(W_-)=1.
\end{equation}

\begin{lem}\label{lemma:cz}
Let $x\in\mathcal{P}(\theta_t)$ be a non-degenerate contractible periodic orbit and let $u\colon D^2\to M$ be a capping of $x$. Then
\begin{equation}\label{eq:CZindex}
\cz(\psi_t,x,u)
=
\sum_{j=1}^n[[\theta_j(\Psi_x)]].
\end{equation}
\end{lem}

\begin{proof}
Let $\Phi$ be a path of non-degenerate symplectic matrices connecting $\Psi_x(1)$ to $W_+$ or $W_-$, according to the sign of $\det(I-\Psi_x(1))$. By additivity under concatenation,
\begin{equation}\label{eq:cz2}
\cz(\psi_t,x,u)
=
\Delta(\Psi_x^\prime)
=
\Delta(\Psi_x*\Phi)
=
\Delta(\Psi_x)+\Delta(\Phi).
\end{equation}

Suppose that $\det(I-\Psi_x(1))>0$, so that $\Phi$ connects $\Psi_x(1)$ to $W_+$. By \eqref{eq:mean index2} and \eqref{eq:theta},
\begin{equation}\label{eq:deltaphi}
\Delta(\Phi)
=
\sum_{j=1}^n
\bigl(\theta_j(W_+)-\theta_j(\Psi_x(1))\bigr)
=
\sum_{j=1}^n
\bigl(1-\theta_j(\Psi_x(1))\bigr).
\end{equation}
By our choice of determinations, the numbers of indices for which $\theta_j(\Psi_x(1))=0$ and $\theta_j(\Psi_x(1))=2$ coincide, so their contributions cancel. For $\theta_j(\Psi_x(1))\in(0,2)$,
\[
1-\theta_j(\Psi_x(1))
=
[[\theta_j(\Psi_x(1))]]-\theta_j(\Psi_x(1)).
\]
Combining this identity with \eqref{eq:cz2}, \eqref{eq:deltaphi}, and \eqref{eq:mean index} yields:
\begin{align*}
\cz(\psi_t,x,u)=\sum_{j=1}^n([[\theta_j(\Psi_x(1))]] + (\theta_j(\Psi_x)-\theta_j(\Psi_x(1))))= \sum_{j=1}^n[[\theta_j(\Psi_x)]].
\end{align*}
The last equality follows from the $2$-periodicity of the function $x\mapsto x-[[x]]$.

If $\det(I-\Psi_x(1))<0$, then $\Phi$ connects $\Psi_x(1)$ to $W_-$, and
\[
\Delta(\Phi)
=
\sum_{j=1}^n
\bigl(\theta_j(W_-)-\theta_j(\Psi_x(1))\bigr)
=
\sum_{j=2}^n
\bigl(1-\theta_j(\Psi_x(1))\bigr),
\]
since $\theta_1(W_-)=\theta_1(\Psi_x(1))=0$. The same argument applies since the number of indices $j\geq2$ for which $\theta_j(\Psi_x(1))=0$ equals the number of indices for which $\theta_j(\Psi_x(1))=2$.
\end{proof}

\subsection{Properties} We recall some useful properties of the Conley--Zehnder index and the mean index. Let $x \in \mathcal{P}(\theta_t)$ be a periodic orbit and $u\colon D^2 \to M$ a capping of $x$. The $k$-th iteration of $x\colon \R/\Z \to M$ is the orbit $x^k \colon [0,k] \to M,\ x^k(t)=x(t)$. We denote by $u^k$ the natural capping of $x^k$ obtained from $u$. 
\begin{enumerate}
	\item({Capping change}) Let $u$ and $v$ be two cappings of a periodic orbit $x$. We have $$\Delta(\psi_t,x,u) = \Delta(\psi_t,x,v) + 2\langle c_1(TM), v \# (-u) \rangle,$$ and if $x$ is non-degenerate, we have $$\cz(\psi_t,x,u) = \cz(\psi_t,x,v) + 2\langle c_1(TM), v \# (-u) \rangle,$$ where $c_1(TM) \in \h^2(M;\Z)$ denotes the usual Chern class. \label{cappingdep} \\
	In particular, when $(M,\omega)$ is symplectically Calabi--Yau, i.e., $c_1(TM)|_{\pi_2(M)} = 0$, both indices are well-defined and independent of the capping \(u\).
\item (Grading) Assume that $(M,\om)$ is symplectically Calabi--Yau and that all orbits in $\mathcal{P}(\theta_t)$ are non-degenerate. Then the Floer--Novikov complex admits an absolute $\Z$-grading given by the Conley--Zehnder index. In particular, $\hfn_k(\psi_t)$ is generated over $\Lambda_{\om,\theta}$ by classes $[x,u]$ with $\cz(\psi_t,x,u)=k,$ and the differential decreases the grading by one. \label{grading}
	\item (Iteration) $\Delta(\psi_t^k,x^k,u^k) = k \Delta(\psi_t,x,u)$. \label{iteformula}
	\item (Indices Relation) For a non-degenerate periodic orbit $x \in \mathcal{P}(\theta_t)$ and a capping $u$ of $x$, we have \label{reldeltacz}
	\begin{equation*}
		\vert \cz(\psi_t,x,u) - \Delta(\psi_t,x,u) \vert < n.
	\end{equation*}
	\item (Change of isotopy) Let $\psi_t$ and $\phi_t$ be two symplectic isotopies based at the identity such that $\psi_1=\phi_1=\psi,$ and let $x \in  \fix(\psi)$. Consider $\tau_\psi$ and $\tau_\phi$ trivializations of $(\psi_t(x))^*TM$ and $(\phi_t(x))^*TM$, respectively. The concatenation $\Gamma_x := \Psi_{x,\tau_\psi}*\overline{\Phi_{x,\tau_\phi}}$, where the bar denotes the reversed path, is a loop in $\mathrm{Sp}(2n)$ and we have
$$\Delta(\psi_t,x,\tau_\psi) - \Delta(\phi_t,x,\tau_\phi) = \Delta(\Gamma_x) = 2\mu(\Gamma_x),$$
where $\mu(\Gamma_x)\in\mathbb Z$ denotes the Maslov index of the loop $\Gamma_x$.

In particular,
$$\Delta(\psi_t,x,\tau_\psi) \equiv \Delta(\phi_t,x,\tau_\phi) \pmod{2}$$
and thus the mean index modulo $2$ is independent of the choice of the symplectic isotopy and of the trivialization. \label{changeisotopy}
\end{enumerate}

\begin{remark}\label{rem:rank}
Using the grading in \eqref{grading}, the proof of Theorem \ref{thm:hfnandhn} yields the equality
$$\mathrm{rank}_{\Lambda_{\om,\theta}} \hfn_k(\psi_t) = \dim_{\Lambda_\theta} \hn_{k+n}(\theta),$$
where $\dim M = 2n$. 
\end{remark}


\subsection{Even integer mean index} We state the estimate used to obtain our main results.

\begin{thm}\label{thm:zeromean}
Suppose that $x_0 \in \fix_0(\psi)$ is a fixed point such that $\Delta(\psi_t,x_0)$ is an even integer. Then 
$$|\cz(\psi_t, x_0)-\Delta(\psi_t,x_0)| \leq n-2.$$ In particular, $\cz(\psi_t, x_0) = \Delta(\psi_t,x_0)$ when $n=2$.
\end{thm}

To prove this, we need two lemmata concerning the function $[[\cdot]]$ defined above.

\begin{lem}\label{lem1}
	Let $\alpha$ and $\beta$ be two real numbers. Then 
$$|[[\alpha+\beta]]-[[\alpha]]-[[\beta]]|\leq 1.$$
\end{lem}
\begin{proof}
	Define $\delta_x:=x-[[x]]$, which satisfies $|\delta_x|\in[0,1)$ for all $x\in\R$ and $\delta_x=0$ if and only if $x\in\Z$. Observe that 
	\begin{equation*}
		\Delta:=|[[\alpha+\beta]]-[[\alpha]]-[[\beta]]|=|\delta_{\alpha+\beta}-\delta_\alpha-\delta_\beta|\leq2,
	\end{equation*}
	since $\Delta$ is an integer and smaller than $3$ by the triangle inequality. It is therefore enough to prove that $\Delta\neq2$. If at least one of the values $\alpha$, $\beta$, and $\alpha +\beta$ is an integer, then $\Delta<2$ since one of the corresponding $\delta$-values must vanish. Otherwise, $[[\alpha+\beta]]$, $[[\alpha]]$, and $[[\beta]]$ are all odd integers. Therefore, $\Delta$ is odd, and, hence, $\Delta=1$.
\end{proof}

\begin{lem}\label{lemma:meanzero}
	Let $\alpha_i$, $i=1,\ldots,n$ ($n\geq 2$), be real numbers that add up to zero. Then $$\sum_{i=1}^{n}[[\alpha_i]]\in[-n+2,n-2].$$
\end{lem}
\begin{proof}
	When $n=2$, the statement follows from the fact that $[[\cdot]]$ is an odd function. We prove the statement by induction on $n$.
	\begin{eqnarray*}
		\left|\sum_{i=1}^{n}[[\alpha_i]]\right|&=& \left|\sum_{i=1}^{n-1}[[\alpha_i]] - \left[\left[(\sum_{i=1}^{n-1}\alpha_i)\right]\right]\right| \\
		&\leq&  \left| \sum_{i=1}^{n-2}[[\alpha_i]] + [[\alpha_{n-1}]] - \left(\left[\left[\sum_{i=1}^{n-2}\alpha_i\right]\right] +[[\alpha_{n-1}]]\right)\right|\\ &+&  \left| \left(\left[\left[\sum_{i=1}^{n-2}\alpha_i\right]\right]+[[\alpha_{n-1}]]\right) -\left[\left[(\sum_{i=1}^{n-1}\alpha_i)\right]\right]\right| \\
		&=&  \left| \sum_{i=1}^{n-2}[[\alpha_i]] + \left(\left[\left[-\sum_{i=1}^{n-2}\alpha_i\right]\right]\right)\right|\\ &+&  \left| \left(\left[\left[\sum_{i=1}^{n-2}\alpha_i\right]\right]+[[\alpha_{n-1}]]\right) -\left[\left[(\sum_{i=1}^{n-1}\alpha_i)\right]\right]\right| \\
		&\leq& (n-3)+1=n-2. 
	\end{eqnarray*}
	The last inequality follows firstly by the induction step and secondly by Lemma~\ref{lem1}.
\end{proof}

Finally, we prove Theorem \ref{thm:zeromean}.

\begin{proof}[Proof of Theorem \ref{thm:zeromean}]

By assumption the mean index of $x_0$ is an even integer and  by  identity \eqref{eq:mean index}, we have
$
\Delta(\psi_t,x_0) = \Delta(\Psi_{x_0})= \sum_{j=1}^n \theta_j(\Psi_{x_0})=2m,$ for some $m \in \Z$.
Let  $\Phi\colon[0,1]\to\mathrm{Sp}(2n)$ be a loop such that 
\begin{equation}\label{eq:vanishing}
 0= \Delta(\Phi*\Psi_{x_0})=\sum_{j=1}^{n}\theta_j(\Phi*\Psi_{x_0})=\sum_{j=1}^{n}(\theta_j(\Psi_{x_0}) + 2m_j).
 \end{equation}
It follows that $\sum_{j=1}^n 2m_j = -2m$. By Lemma \ref{lemma:cz},
\begin{align*}
\cz(\psi_t,x_0) = \sum_{j=1}^n [[\theta_j(\Psi_{x_0})]]&= \sum_{j=1}^n ([[\theta_j(\Psi_{x_0})+2m_j]] -2m_j)\\
	&=\sum_{j=1}^n [[\theta_j(\Psi_{x_0})+2m_j]] +2m,
\end{align*}
where in the second equality we use the $2$-periodicity of the function $x\mapsto x - [[x]]$.
Finally, by \eqref{eq:vanishing} and Lemma \ref{lemma:meanzero},
$$\sum_{j=1}^n [[\theta_j(\Psi_{x_0})+2m_j]]\in[-(n-2),n-2]$$
and this finishes the proof.
\end{proof}


\section{Proofs}\label{sec:proof}
\subsection{Proof of Theorem \ref{thm:main}}
There are only finitely many fixed points, since $\psi$ is non-degenerate. Let
\begin{equation}\label{eq:delta}
0<\delta<\min \{|\Delta(\psi_t,x)| ;x\in\fix_0(\psi_t), \Delta(\psi_t,x)\neq0\}
\end{equation}
and let  $p_0$ be a prime number such that $\psi^p$ is non-degenerate for all primes $p\geq p_0$ and $p_0\delta>2n$. Suppose that $\fix_0(\psi_t^p)=\fix_0(\psi_t)$ for a prime number $p\geq p_0$.

By assumption \eqref{eq:main} and Remark \ref{rem:rank} there exists a non-trivial class $\eta_p\in\hfn(\psi^p_t)$ of degree $k\in\{-n,-n+1,n-1,n\}$. It follows from the indices relation in \eqref{reldeltacz} that any $y^{(p)}\in\fix_0(\psi^p_t)$, such that a capping of $\{\psi^p_t(y^{(p)})\}$ appears as a summand in a representative of $\eta_{p}$, has mean index $\Delta(\psi^p_t,y^{(p)})\in(-2n,2n)$. Moreover,  $\Delta(\psi^p_t,y^{(p)})\not=0$, since otherwise $\Delta(\psi^p_t,y^{(p)})=0$ and, by Theorem~\ref{thm:zeromean}, $|\cz(\psi^p_t,y^{(p)})|\leq n-2$, which contradicts $k\in\{-n,-n+1,n-1,n\}$.   Hence $\Delta(\psi^p_t,y^{(p)})\in(-2n,2n)\setminus \{0\}$. However, $$|\Delta(\psi^p_t,y^{(p)})|=p|\Delta(\psi_t,y)|>p\delta>2n,$$ which is a contradiction. \qedsymbol


\subsection{Proof of Theorem \ref{thm:t4}} By composing a given isotopy $\psi_t$ based at the identity and such that $\psi_1 = \psi$ with a loop of symplectomorphisms if necessary, we may assume that $x_0$ is contractible without changing the parity of the mean index; see property \eqref{changeisotopy}. The Hamiltonian case follows from Remark \ref{rmk:Hamiltoniancase}. Therefore, we may assume $\flux(\psi_t) \neq 0$. In this case, we have $\hfn (\psi^k_t) = 0$, for all $k \in \mathbb{N}$. {Recall that the mean index of $x_0$ is an even integer, hence, by }Theorem \ref{thm:zeromean}, we have $\cz(\psi_t,y) = \Delta(\psi_t,y)$ for every fixed point $y$ with even mean index equal to $\Delta(\psi_t,x_0)$. Since $\hfn (\psi_t) = 0$, there must exist at least one fixed point $x_1$ such that $\vert \cz(\psi_t,x_1)-\cz(\psi_t,x_0) \vert = 1$, and hence, $\Delta(\psi_t,x_1)\neq\Delta(\psi_t,x_0)$. Let
\begin{equation}\label{eq:delta}
0<\delta<\min \left\{\vert \Delta(\psi_t,x)-\Delta(\psi_t,x_0) \vert\,;x\in\fix_0(\psi_t), \Delta(\psi_t,x)\neq\Delta(\psi_t,x_0)\right\}.
\end{equation}

Let $p_0$ be a prime number such that $\psi^p$ is non-degenerate for all prime $p\geq p_0$ and $p_0 \delta > 3$, where $\delta>0$ is defined as in \eqref{eq:delta}. 

Suppose that $\fix_0(\psi_t)=\fix_0(\psi_t^p)$ for a prime number $p\geq p_0$. Since $\delta p_0 >3$, the indices relation \eqref{reldeltacz} implies there are no orbits $\psi_t^p(y^{(p)})$ with $\vert \cz(\psi_t^p,y^{(p)})-\cz(\psi_t^p,x_0^{(p)}) \vert = 1$. This contradicts $\hfn(\psi_t^p) = 0$. Therefore, $\fix_0(\psi_t)\subsetneq\fix_0(\psi_t^p)$. \qedsymbol
	
\enlargethispage{2\baselineskip}
\bibliographystyle{alpha}
\bibliography{biblio}

\blfootnote{AMS subject codes: 53D40, 37J99}
\end{document}